\documentclass[11pt]{amsart}

\usepackage{amsfonts,amssymb,amsthm,eucal,color,bbm,verbatim,graphicx}
\usepackage[french,english]{babel}

\usepackage[margin=1.25in]{geometry}
\calclayout

\newtheorem{thm}{Theorem}

\newtheorem{lemma}[thm]{Lemma}
\newtheorem{prop}[thm]{Proposition}

\theoremstyle{definition}

\newcommand{\E}{\mathbb{E}}

\newcommand{\Prob}{\mathbb{P}}

\newcommand{\abs}[1]{\left\vert #1 \right\vert}

\newcommand{\eps}{\varepsilon}

\DeclareMathOperator{\var}{Var}

\newcommand{\Unitary}[1]{\mathbb{U}\left(#1\right)}

\newcommand{\ind}[1]{\mathbbm{1}_{#1}}

\allowdisplaybreaks[3]

\author{Elizabeth S.\ Meckes}

\author{Mark W.\ Meckes}

\address{Department of Mathematics, Applied Mathematics, and
  Statistics, Case Western Reserve University, 10900 Euclid Ave.,
  Cleveland, Ohio 44106, U.S.A.}

\email{elizabeth.meckes@case.edu}

\address{Department of Mathematics, Applied Mathematics, and
  Statistics, Case Western Reserve University, 10900 Euclid Ave.,
  Cleveland, Ohio 44106, U.S.A.}

\email{mark.meckes@case.edu}

\title[Rate of convergence for spectral measures of random unitary
matrices]{A sharp rate of convergence for the empirical spectral
  measure of a random unitary matrix}

\keywords{Random matrices; empirical spectral measures; determinantal
  point processes}

\begin{document}

\maketitle

\begin{abstract}
  We consider the convergence of the empirical spectral measures of
  random $N \times N$ unitary matrices. We give upper and lower bounds
  showing that the Kolmogorov distance between the spectral measure
  and the uniform measure on the unit circle is of the order
  $\log N / N$, both in expectation and almost surely. This implies in
  particular that the convergence happens more slowly for Kolmogorov
  distance than for the $L_1$-Kantorovich distance. The proof relies
  on the determinantal structure of the eigenvalue process.
\end{abstract}

Let $U\in\Unitary{N}$ be a random matrix, distributed according to
Haar measure.  Denote the eigenvalues of $U$ by
$e^{i\theta_1},\ldots,e^{i\theta_N}$, and let $\mu_N$ denote the
empirical spectral measure of $U$; that is,
\[
\mu_N:=\frac{1}{N}\sum_{j=1}^N\delta_{e^{i\theta_j}}.
\]
It is easy to see by symmetry that that $\E \mu_N = \nu$ for every
$N$, where $\nu$ is the uniform probability measure on the unit circle
in the complex plane.

The convergence of the empirical spectral measure of a random matrix
to a limiting distribution, as the size of the matrix tends to
infinity, has been studied extensively for a variety of random matrix
ensembles, most notably for Wigner matrices.  In particular, the
empirical spectral measure converges to the semicircle law in the
Komogorov distance at rate $(\log N)^{c}/N$ (see \cite{GoNaTi}).

In the context of random unitary matrices, the convergence of $\mu_N$
to the uniform measure on the circle (weakly, in probability) was
first proved in \cite{DS}. In \cite{HP} a large deviations principle
was proved which in particular shows that convergence occurs with
probability 1.  In earlier work (see \cite{MM-ECP}), we have
quantified this convergence, with respect to the $L_1$-Kantorovich
distance $W_1$.  Specifically, there are absolute constants $C_1$ and
$C_2$ such that
\begin{equation}\label{E:W1bound}
\E W_1(\mu_N,\nu)\le \frac{C_1 \sqrt{\log(N)}}{N},
\end{equation}
and, with probability $1$,
\begin{equation}\label{E:W1-as}
W_1(\mu_N,\nu)\le \frac{C_2 \sqrt{\log(N)}}{N}
\end{equation}
for all sufficiently large $N$.

In this note, we consider instead the Kolmogorov distance
\[
d_K(\mu_N,\nu)=\sup_{0\le
  \theta<2\pi}\abs{\frac{1}{N}\mathcal{N}_\theta-\frac{\theta}{2\pi}},
\]
where $\mathcal{N}_\theta$ is the number of eigenvalues
$e^{i\theta_j}$ of $U$ with $0\le\theta_j\le\theta$.  That is, we are
interested in upper and lower bounds for the supremum of the
stochastic process 
\[
X_\theta:=\abs{\frac{1}{N}\mathcal{N}_\theta-\frac{\theta}{2\pi}}
\]
indexed by $\theta\in[0,2\pi)$.

\begin{thm}
  \label{T:E-K}
  There are universal constants $c_1, c_2, c_3 > 0$ such that 
  \[
  c_1 \frac{\log(N)}{N}\le \E d_K(\mu_N,\nu) \le c_2 \frac{\log(N)}{N}
  \]
  for all $N$, and with probability $1$,
  \[
  d_K (\mu_N,\nu) \le c_3 \frac{\log(N)}{N}
  \]
  for all sufficiently large $N$.
\end{thm}

After the first version of this paper was written, we were informed by
Paul Bourgade of the results of \cite{ArBeBo}, which in particular
show convergence in probability of $d_K (\mu_N, \nu)$. Combining the
results of \cite{ArBeBo} with our methods, we prove the following
improvement of the first part of Theorem \ref{T:E-K}.

\begin{thm}
  \label{T:Lp}
  For every $p > 0$,
  \[
  \frac{N}{\log N} d_K(\mu_N, \nu) \xrightarrow{L_p} \frac{1}{\pi}
  \]
  as $N \to \infty$.
\end{thm}

One interesting consequence of the theorems together with the bounds
\eqref{E:W1bound} and \eqref{E:W1-as} proven in \cite{MM-ECP} is that
in this setting, the expected rate of convergence of $\mu_N$ to $\nu$
in the $L_1$-Kantorovich distance is strictly faster than the expected
rate of convergence in the Kolmogorov distance. This is in contrast to
the setting of more classical limit theorems, for which the rates are
often of the same order; e.g., for i.i.d.\ samples, the rate is
$N^{-1/2}$ in both metrics.

While it is desirable to have results comparable to \eqref{E:W1bound}
and \eqref{E:W1-as} for the more familiar and widely used Kolmogorov
metric, the interest stems in large part from the connection between
Kolmogorov bounds and maximal eigenvalue spacing; a large gap between
successive eigenvalues corresponds to a large arc to which the
spectral measure assigns no mass.  There is great interest in the
asymptotics of the maximal eigenvalue spacing for random unitary
matrices, in part because of the connection to the Riemann zeta
function; the distribution of the maximal eigenvalue spacings for
$N\times N$ random unitary matrices are conjectured to predict the
statistics of spacings between successive zeroes of the zeta function
at height $T$ along the critical line, when
$N\approx\log\left(\frac{T}{2\pi}\right)$.  A significant recent
contribution on the maximal eigenvalue spacing was made in \cite{BAB},
where it was shown that if $\mathcal{T}^{(N)}$ is the maximum
eigenvalue gap of a uniform $U \in \Unitary{N}$, then
\[
\frac{N}{\sqrt{32\log(N)}}\,\mathcal{T}^{(N)}\xrightarrow{L_p}1
\]
for all $p>0$.  This implies in particular that
$\E d_K(\mu_N,\nu)\ge\frac{c\sqrt{\log (N)}}{N};$ Theorems \ref{T:E-K}
and \ref{T:Lp}
shows that the correct rate is in fact $\frac{\log (N)}{N}$.

\bigskip

A crucial property underpinning the proofs of the theorems is that the
eigenvalue angles $\theta_1,\ldots,\theta_N$ are a determinantal point
process on $[0,2\pi]$, with symmetric kernel
\[
K_N(x,y)=\frac{\sin\left(\frac{N(x-y)}{2}\right)}{\sin\left(\frac{x-y}{2}\right)}
\]
(see \cite [chapter 11]{Mehta}).  In particular, the following
properties of the eigenvalue counting function are consequences of the
d.p.p.\ structure.

\begin{prop}\label{T:dpp-properties}~

\begin{enumerate}
\item \label{P:bernoullis} Let $A\subseteq[0,2\pi]$, and let
  $\mathcal{N}_A$ denote the number of eigenvalue angles of $U$ lying
  in $A$.  Then there are independent Bernoulli random variables
  $\xi_1,\ldots,\xi_N$ such that
  \[
  \mathcal{N}_A\overset{d}=\sum_{j=1}^N\xi_j.
  \]

\item \label{P:neg-assoc}The eigenangle process of $U$ is negatively
  associated: if $A,B\subseteq[0,2\pi]$ are disjoint, then
  \[
  \Prob\left[\mathcal{N}_A\ge s,\mathcal{N}_B\ge t\right] \le
  \Prob\left[\mathcal{N}_A \ge s \right] \Prob\left[\mathcal{N}_B\ge
    t\right].
  \]
\end{enumerate}
\end{prop}

The first part of Proposition \ref{T:dpp-properties} follows from the
corresponding property for a quite general class of determinantal
point process, due to Hough--Krishnapur--Peres--Vir\'ag \cite[Theorem
7]{HKPV}.  The second part is again a consequence of a more general
statement about determinantal point processes, this time due to Ghosh
\cite[Theorem 1.4]{Ghosh}.

The representation of the counting function as a sum of independent
Bernoulli random variables is a powerful tool; it opens the doors to
countless results of classical probability. (For other uses of this
idea in the theory of random unitary matrices, see
\cite{MM-ECP,MM-DA}; see also
\cite{Dallaporta1,Dallaporta2,MM-Toulouse} for related approaches in
other random matrix ensembles.) We will be particularly interested in
the tail probabilities
\[
\Prob\left[\mathcal{N}_I - \E\mathcal{N}_I > t\right],
\]
for $t > 0$ and $I$ an interval to be specified; note that by rotation
invariance, this is equal to
$\Prob\left[\mathcal{N}_\theta - \E\mathcal{N}_\theta > t\right]$,
where $\theta$ is the length of $I$.  In the classical setting of a
sum of independent random variables, an upper bound on such tail
probabilities is given by Bernstein's inequality (see e.g.\ \cite[Lemma
4.3.4]{Talagrand}), while a lower bound was proved by Kolmogorov (see
\cite [Hilfssatz IV]{Kol}).

\begin{prop}\label{T:tail-bounds}
  Let $X_1,\ldots,X_n$ be independent random variables, with $|X_j|\le
  M$ almost surely, for each $j$.  Let
  \[
  S_n:=\sum_{j=1}^nX_j\qquad\qquad s_n^2=\var(S_n).
  \]
  Then
  \begin{enumerate}
  \item \label{P:Bernstein} for all $x>0,$
    \begin{equation*}    
      \Prob\left[S_n- \E S_n> xs_n \right]
      \le  \exp \left(-\min\left\{\frac{x^2}{4},
          \frac{xs_n}{2M}\right\}\right),
    \end{equation*}
    and
  \item \label{P:Kolmogorov}if $x\ge 512$ and
    $a:=\frac{xM}{s_n}\le\frac{1}{256}$, then for
    $\eps=\max\left\{64\sqrt{a},\frac{32\sqrt{\log(x^2)}}{x}\right\}$,
    \[
    \Prob\left[S_n-\E S_n>xs_n\right]\ge
    e^{-\frac{x^2}{2}(1+\eps)}.
    \]
  \end{enumerate}
\end{prop}

By part \ref{P:bernoullis} of Proposition \ref{T:dpp-properties}, the
conclusions of Proposition \ref{T:tail-bounds} apply to the counting
functions $\mathcal{N}_I$ with $M=1$; for them to give usable
estimates, formulae (or at least asymptotics) for the means and
variances of the counting functions are needed.  The mean is trivial
to compute by symmetry. Rather precise asymptotics can be determined
for the variance, as a further application of the determinantal point
process structure of the ensemble of eigenvalues.  The estimates in
the following lemma were proved in \cite{MM-ECP,MM-DA}.

\begin{lemma}\label{T:means-vars}~

\begin{enumerate}
\item \label{P:mean}For $\theta\in[0,2\pi]$, 
  \[
  \E\mathcal{N}_\theta=\frac{N\theta}{2\pi}.
  \]
\item \label{P:var1}For $\theta \in [0, 2\pi]$,
  \[
  \var \mathcal{N}_\theta \le \log (eN).
  \]

\item \label{P:var2}If $\frac{3 \pi}{2N} \le \theta \le \frac{\pi}{2}$,
  \[
  \frac{1}{3\pi^2} \log \left(\frac{2 N\theta}{3\pi}\right)\le \var
  \mathcal{N}_\theta \le \frac{1}{2} \log \bigl(e^{3/2} N\theta
  \bigr).
  \]
  
\end{enumerate}
\end{lemma}

\bigskip

With these ingredients in place, we now turn to upper and lower bounds
on $\E d_{K}(\mu_n,\nu)$.  

\begin{proof}[Proof of Theorem \ref{T:E-K}]
  We consider the upper bounds first.  If
  $\frac{2\pi k}{N} \le \theta < \frac{2\pi (k+1)}{N}$, then
  \[
  \mathcal{N}_\theta-\tfrac{N\theta}{2\pi} \le
  \mathcal{N}_{\frac{2\pi (k+1)}{N}} - (k+1) + 1
  \]
  and
  \[
  \mathcal{N}_\theta-\tfrac{N\theta}{2\pi} \ge
  \mathcal{N}_{\frac{2\pi k}{N}} - k - 1,
  \]
  so that 
  \begin{equation}\label{E:reduction}
    d_K(\mu_N,\nu) \le \frac{1}{N} \sup_{1 \le k \le N}
    \abs{\mathcal{N}_{\frac{2\pi k}{N}}-k} + \frac{1}{N}.
  \end{equation}

  As discussed above, Proposition \ref{T:tail-bounds} can be applied
  to the counting function $\mathcal{N}_{\frac{2\pi k}{N}}$. Part
  \ref{P:Bernstein} of Proposition \ref{T:tail-bounds}) and part
  \ref{P:var1} of Lemma \ref{T:means-vars} imply that
  \begin{equation}
    \label{E:P-ub}
    \begin{split}
      \Prob\left[\sup_{1\le k\le N} \abs{\mathcal{N}_{\frac{2\pi k}{N}}-k}
        > x\right] &\le \sum_{k=1}^{N}
      \Prob\left[\abs{\mathcal{N}_{\frac{2\pi k}{N}} - k}
        > x\right]\\
      &\le 2 \sum_{k=1}^{N}
      \exp\left(-\min\left\{\frac{x^2}{4\var(\mathcal{N}_{\frac{2\pi k}{N}})},\frac{x}{2}\right\}\right)
      \\
      &\le 2 N\exp\left(-\min\left\{\frac{x^2}{4\log(e
            N)},\frac{x}{2}\right\}\right).
    \end{split}
  \end{equation}
  Now, since $\abs{\mathcal{N}_{\frac{2\pi k}{N}}-k} \le N$ for all
  $k$, it follows from the estimate above that 
  for any $x > 0$,
  \begin{equation*}
    \begin{split}
      \E \left[\sup_{1\le k\le N} \abs{\mathcal{N}_{\frac{2\pi k}{N}} - k}
      \right]
      & \le x + N \Prob\left[\sup_{1\le k\le N} \abs{
          \mathcal{N}_{\frac{2\pi k}{N}} -k} > x \right] \\
      & \le x + 2N^2 \exp\left(-\min\left\{\frac{x^2}{4\log(e
            N)},\frac{x}{2}\right\}\right).
    \end{split}
  \end{equation*}
  Setting $x=4\log(eN)$, this implies
  \[
  \E \left[\sup_{1\le k\le N} \abs{\mathcal{N}_{\frac{2\pi k}{N}} - k}
  \right]
  \le 4 \log(eN) + \frac{2}{e^2}.
  \]
  The claimed upper bound on $\E d_K(\mu_N,\nu)$ now follows from
  \eqref{E:reduction}.
  
  \medskip
  
  Setting $x = 6 \log(eN)$ in \eqref{E:P-ub} yields
  \[
  \Prob\left[\sup_{1\le k\le N} \abs{\mathcal{N}_{\frac{2\pi k}{N}}-k}
    > 6 \log(eN) \right] \le \frac{2}{e^3 N^2}.
  \]
  The almost sure rate of convergence now follows from
  \eqref{E:reduction} and the Borel--Cantelli lemma.

  \medskip
  
  For the lower bound, note first that given probability measures
  $\mu$ and $\nu$ on $[0,2\pi)$,
  \[
  d_K(\mu,\nu)\le\sup_{0\le a\le
    b<2\pi}\big|\mu((a,b])-\nu((a,b])\big|\le 2d_K(\mu,\nu).
  \]
  Let $\mathcal{I}_N$ be a collection of $T$ disjoint subintervals of
  $[0,2\pi)$, each of length $N^{-1/2}$; in particular, $T\le 2\pi
  \sqrt{N}$.  Then by the Bonferroni inequalities,
  \begin{equation*}
    \begin{split}
      \Prob\left[d_K(\mu_N,\nu)>x\right] &\ge
      \Prob\left[\sup_{I\in\mathcal{I}_N}(\mu_N(I)-\nu(I))>2x\right]\\
      &\ge\sum_{I\in\mathcal{I}_N}\Prob\big[(\mu_N(I)-\nu(I))>2x\big]\\
      &\qquad\qquad-\frac{1}{2}\sum_{\substack{I,
          J\in\mathcal{I}_N\\
          I\neq J}}\Prob\big[(\mu_N(I)-\nu(I)),
      (\mu_N(J)-\nu(J))>2x\big]\\
      &\ge\sum_{I\in\mathcal{I}_N}\Prob\big[(\mu_N(I)-\nu(I))>2x\big]\\
      &\qquad\qquad-\frac{1}{2}\sum_{\substack{I,
          J\in\mathcal{I}_N\\
          I\neq
          J}}\Prob\big[(\mu_N(I)-\nu(I))>2x\big]
      \Prob\big[(\mu_N(J)-\nu(J))>2x\big],
    \end{split}
  \end{equation*}
  where the last estimate follows from the negative association
  property of part \ref{P:neg-assoc} of Proposition
  \ref{T:dpp-properties}.  Since all of the intervals $I$ have the
  same length, it follows from the rotation-invariance of both
  measures that this last expression is exactly
  \begin{equation*}
    TP-\frac{T(T-1)}{2}P^2\ge \frac{1}{2}TP(2-TP),
  \end{equation*}
  where $P$ is the common value of $\Prob\big[(\mu_N(I)-\nu(I))>2x\big]$
  for $I\in\mathcal{I}_N$.  It follows that if $x$ and $T$ can be
  chosen such that $TP \in \left[\frac{1}{2},\frac{3}{2}\right]$, then
  \begin{equation*}
    \Prob\left[d_K(\mu_N,\nu)>x\right] \ge \frac{3}{8},
  \end{equation*}
  and therefore
  \begin{equation}
    \label{E:x-lb}
    \E d_K(\mu_N,\nu) \ge \frac{3}{8} x.
  \end{equation}
  
  Since each $I$ has length $N^{-1/2}$, by we have that for sufficiently
  large $N$,
  \begin{equation*}
    \begin{split}
      P = \Prob\left[\mu_N(I)-\nu(I) > \frac{\var
          (\mathcal{N}_{I})}{256 N} \right] & =
      \Prob\left[\mathcal{N}_{I} - \frac{|I|}{2\pi} >
        \frac{\var (\mathcal{N}_{I})}{256} \right] \\
      & \ge \exp\left(-\frac{5}{2^{17}} \var(\mathcal{N}_{I})\right) \\
      & \ge \exp\left(-\frac{5}{2^{19}} \log \bigl(e^3
        N\bigr)\right);
    \end{split}
  \end{equation*}
  the first estimate follows from part \ref{P:Kolmogorov} of
  Proposition \ref{T:tail-bounds} and the second follows from part
  \ref{P:var2} of Lemma \ref{T:means-vars} with $\theta=N^{-1/2}$. It
  follows that for all sufficiently large $N$, $N^{1/2}P \ge 2$, and
  therefore $TP \in \left[\frac{1}{2},\frac{3}{2}\right]$ for some
  integer $1 \le T \le N^{1/2}$. Then by \eqref{E:x-lb} and Lemma
  \ref{T:means-vars},
  \[
  \E d_K(\mu_N,\nu) \ge \frac{3 \var(\mathcal{N}_{N^{-1/2}})}{2^{12}N}
  \ge c \frac{\log(N)}{N}
  \]
  for all $N$ large enough.
\end{proof}

\begin{proof}[Proof of Theorem \ref{T:Lp}]
  In \cite{ArBeBo}, the
  authors state that
  \[
  \frac{1}{\log N} \sup_{0 \le \pi < 2\pi} \left(\mathcal{N}_\theta -
    \frac{N\theta}{2\pi} \right) \to \frac{1}{\pi}
  \]
  in probability. It can similarly be shown \cite{Bourgade} that the
  corresponding infimum converges in probability to $- \frac{1}{\pi}$,
  from which it follows that
  \begin{equation*}
    \frac{N}{\log N} d_K(\mu_N, \nu) \to \frac{1}{\pi}
  \end{equation*}
  in probability.

  For a fixed $\eps > 0$,
  \begin{equation*}
    \begin{split}
      \E \abs{\frac{N}{\log N} d_K(\mu_N, \nu) - \frac{1}{\pi}}^{p}
      & \le \eps^p + \E \abs{\frac{N}{\log N} d_K(\mu_N, \nu) -
        \frac{1}{\pi}}^{p} \ind{\abs{\frac{N}{\log N} d_K(\mu_N, \nu) -
        \frac{1}{\pi}} > \eps} \\
      & \le \eps^p + \sqrt{\E \abs{\frac{N}{\log N} d_K(\mu_N, \nu) -
        \frac{1}{\pi}}^{2p}} \sqrt{\Prob \left[\abs{\frac{N}{\log N} d_K(\mu_N, \nu) -
        \frac{1}{\pi}} > \eps\right]}.
    \end{split}
  \end{equation*}
  The theorem thus follows from the convergence in probability of $\frac{N}{\log N} d_K(\mu_N,
  \nu)$, if we can show that the sequence of random variables
  $\frac{N}{\log N} d_K(\mu_N,\nu)$ is bounded in $L_{2p}$.

 Now, for $x>0$, it follows from \eqref{E:P-ub} that 
  \[
  \E \left[\sup_{1 \le k \le N} \abs{\mathcal{N}_{\frac{2\pi k}{N}} -
      k}\right]^{2p} \le x^{2p} + 2N^{2p+1}
  \exp\left(-\min\left\{\frac{x^2}{4\log(e
        N)},\frac{x}{2}\right\}\right).
  \]
  Choosing $x$ to be a sufficiently large multiple of $\log N$ we
  obtain
  \[
  \E \abs{\sup_{1 \le k \le N} \mathcal{N}_{\frac{2\pi k}{N}} -
    k}^{2p} \le C_p (\log N)^{2p}
  \]
  for some constant $C_p > 0$ depending only on $p$; together with
  \eqref{E:reduction} this implies that
  \[
  \E \left[\frac{N}{\log N} d_K(\mu_N, \nu) \right]^{2p} \le C_p'.
  \qedhere
  \]
\end{proof}

\bibliographystyle{plain}
\bibliography{unitary-kolmogorov}
\end{document}